\documentclass[12pt]{amsart}
\usepackage{amscd,amsmath,amsthm,amssymb}
\usepackage{pstcol,pst-plot,pst-3d}

\usepackage{pstricks}
\usepackage{lineno,xcolor}
\usepackage{pstricks-add}

\usepackage{color}

\newpsstyle{fatline}{linewidth=1.5pt}
\newpsstyle{fyp}{fillstyle=solid,fillcolor=verylight}
\definecolor{verylight}{gray}{0.97}
\definecolor{light}{gray}{0.9}
\definecolor{medium}{gray}{0.85}
\definecolor{dark}{gray}{0.6}

 \def\NZQ{\mathbb}               
 
 \def\QQ{{\NZQ Q}}
 \def\ZZ{{\NZQ Z}}
 \def\RR{{\NZQ R}}

 \def\frk{\mathfrak}               

 \def\mm{{\frk m}}

 \def\G{{\mathcal G}}

 \def\B{{\mathcal B}}
 \def\P{{\mathcal P}}

 \def\ab{{\mathbf a}}
 
 \def\xb{{\mathbf x}}

 \def\cb{{\mathbf c}}

 \def\opn#1#2{\def#1{\operatorname{#2}}} 
 \opn\chara{char} \opn\length{\ell} \opn\pd{pd} \opn\rk{rk}
 \opn\projdim{proj\,dim} \opn\injdim{inj\,dim} \opn\rank{rank}
 \opn\depth{depth} \opn\grade{grade} \opn\height{height}
 \opn\embdim{emb\,dim} \opn\codim{codim}
 
 \opn\Tr{Tr} \opn\bigrank{big\,rank}
 \opn\superheight{superheight}\opn\lcm{lcm}
 \opn\trdeg{tr\,deg}
 \opn\reg{reg} \opn\lreg{lreg} \opn\ini{in} \opn\lpd{lpd}
 \opn\size{size} \opn\sdepth{sdepth}
 \opn\link{link}\opn\fdepth{fdepth}\opn\lex{lex}
 \opn\tr{tr}
 \opn\type{type}
 \opn\Borel{Borel}
 \opn\div{div} \opn\Div{Div} \opn\cl{cl} \opn\Cl{Cl}
 \opn\Spec{Spec} \opn\Supp{Supp} \opn\supp{supp} \opn\Sing{Sing}
 \opn\Ass{Ass} \opn\Min{Min}\opn\Mon{Mon}
 \opn\Ann{Ann} \opn\Rad{Rad} \opn\Soc{Soc}
 \opn\Im{Im} \opn\Ker{Ker} \opn\Coker{Coker} \opn\Am{Am}
 \opn\Hom{Hom} \opn\Tor{Tor} \opn\Ext{Ext} \opn\End{End}
 \opn\Aut{Aut} \opn\id{id}
 
 \opn\nat{nat}
 \opn\pff{pf}
 \opn\Pf{Pf} \opn\GL{GL} \opn\SL{SL} \opn\mod{mod} \opn\ord{ord}
 \opn\Gin{Gin} \opn\Hilb{Hilb}\opn\sort{sort}
 \opn\PF{PF}\opn\Ap{Ap}
 \opn\aff{aff} \opn
 \con{conv} \opn\relint{relint} \opn\st{st}
 \opn\lk{lk} \opn\cn{cn} \opn\core{core} \opn\vol{vol}  \opn\inp{inp} \opn\nilpot{nilpot}
 \opn\link{link} \opn\star{star}\opn\lex{lex}\opn\set{set}
 \opn\width{wd}
 \opn\Fr{F}
 \opn\QF{QF}
 \opn\G{G}
 \opn\type{type}\opn\res{res}
 \opn\gr{gr}
 
 \def\pot#1#2{#1[\kern-0.28ex[#2]\kern-0.28ex]}

 \opn\dirlim{\underrightarrow{\lim}}
 \opn\inivlim{\underleftarrow{\lim}}
 \let\union=\cup
 \let\sect=\cap
 \let\dirsum=\oplus
 \let\tensor=\otimes
 \let\iso=\cong

 \let\Dirsum=\bigoplus
 
 \let\to=\rightarrow
 
 \def\Implies{\ifmmode\Longrightarrow \else
         \unskip${}\Longrightarrow{}$\ignorespaces\fi}
 \def\implies{\ifmmode\Rightarrow \else
         \unskip${}\Rightarrow{}$\ignorespaces\fi}
 \def\iff{\ifmmode\Longleftrightarrow \else
         \unskip${}\Longleftrightarrow{}$\ignorespaces\fi}

 \let\:=\colon
 \newtheorem{Theorem}{Theorem}[section]
 \newtheorem{Lemma}[Theorem]{Lemma}
 \newtheorem{Corollary}[Theorem]{Corollary}
 \newtheorem{Proposition}[Theorem]{Proposition}

 \let\epsilon\varepsilon
 \let\kappa=\varkappa
 \def\qed{\ifhmode\textqed\fi
       \ifmmode\ifinner\quad\qedsymbol\else\dispqed\fi\fi}
 \def\textqed{\unskip\nobreak\penalty50
        \hskip2em\hbox{}\nobreak\hfil\qedsymbol
        \parfillskip=0pt \finalhyphendemerits=0}
 \def\dispqed{\rlap{\qquad\qedsymbol}}
 
 \opn\dis{dis}
 \def\pnt{{\raise0.5mm\hbox{\large\bf.}}}
 
 \opn\Lex{Lex}

\begin{document}
\title {The relevance of Freiman's theorem for combinatorial commutative algebra}

\author {J\"urgen Herzog, Takayuki Hibi and  Guangjun Zhu$^{^*}$}

\address{J\"urgen Herzog, Fachbereich Mathematik, Universit\"at Duisburg-Essen, Campus Essen, 45117
Essen, Germany} \email{juergen.herzog@uni-essen.de}

\address{Takayuki Hibi, Department of Pure and Applied Mathematics, Graduate School of Information Science and Technology,
Osaka University, Toyonaka, Osaka 560-0043, Japan}
\email{hibi@math.sci.osaka-u.ac.jp}

\address{Guangjun Zhu, School of Mathematical Sciences, Soochow University,
 Suzhou 215006, P. R. China}\email{zhuguangjun@suda.edu.cn}

\dedicatory{ }

\begin{abstract}
Freiman's theorem gives a lower bound for  the cardinality of the doubling of a finite set in $\RR^n$. In this paper we give an interpretation of  his theorem for monomial ideals and their fiber cones. We call a quasi-equigenerated monomial ideal a Freiman ideal, if the  set of its exponent vectors  achieves Freiman's  lower bound for its doubling. Algebraic characterizations of Freiman ideals are given,  and finite simple graphs are classified whose edge ideals or matroidal ideals of its cycle matroids are Freiman ideals.
\end{abstract}

\thanks{* Corresponding author.}




\maketitle

\setcounter{tocdepth}{1}

\section*{Introduction}

Let $X$ be a finite subset of $\ZZ^n$, and let $A(X)$ be the affine hull of the set $X$, that is, the smallest affine subspace of $\QQ^n$ containing
$X$.   The doubling of $X$ is the set $2X=\{a+b\: a,b\in X\}$. The starting point of this paper is the following celebrated theorem of Freiman \cite{F1}:
\begin{eqnarray}
\label{freimaninequality}
|2X|\geq (d+1)|X|-{d+1 \choose 2},
\end{eqnarray}
where $d$ is the dimension of the affine space $A(X)$.

Now let $K$ be a field and let $I$ be a graded ideal in the polynomial ring $S=K[x_1,\ldots,x_n]$. We denote by  $\mu(I)$  the minimal  number of generators of $I$  and by $\ell(I)$ the  analytic spread of $I$, that is, the Krull dimension of the fiber cone
$F(I)=\dirsum_{k\geq 0}I^k/\mm I^k$. Here $\mm=(x_1,\ldots,x_n)$ is the graded maximal ideal of $S$. It has been noticed in \cite[Theorem 1.9]{HMZ} that Freiman's theorem  has an interesting consequence regarding the minimal  number of generators of the square of a monomial ideal. Namely, it was shown that if $I\subset S$ is a monomial ideal with the property that  all generators of $I$ have the same degree, then  $\mu(I^2)\geq \ell(I)\mu(I)-{\ell(I)\choose 2}$.

For the application of Freiman's theorem to a  monomial ideal $I\subset S$  and its fiber cone $F(I)$, it is sufficient to require that $I$  is quasi-equigenerated, by which we mean that the exponent vectors of all generators of $I$  lie in a hyperplane of $\ZZ^n$. This guarantees that the set of  exponent vectors of the monomials  $u\in G(I^2)$  is obtained from the set of exponent vectors of the monomials  $v\in G(I)$ by doubling. Here,  $G(I)$ denotes the unique minimal set of monomial generators of  a monomial ideal~$I$.

We call a  quasi-equigenerated monomial ideal $I\subset S$  a Freiman ideal,   if $\mu(I^2)= \ell(I)\mu(I)-{\ell(I)\choose 2}$. Of course,  $I$ is Freiman if and only if equality holds in (\ref{freimaninequality})  for the set $X$ of exponent vectors of the  monomials  $u\in G(I)$. The sets $X\subset \ZZ^n$ for which $|2X|= (d+1)|X|-{d+1 \choose 2}$ are characterized by Stanescu in \cite{Sta}. Here,  in Theorem~\ref{characterization}, we give  several equivalent conditions for a  quasi-equigenerated monomial ideal $I\subset S$ to be a Freiman ideal.  For example, it is shown that $I$ is Freiman if and only if the fiber cone  $F(I)$ of $I$ is Cohen--Macaulay and the  defining ideal of $F(I)$ has a $2$-linear resolution. This homological characterization of Freiman ideals can be deduced from the  fact  that for a Freiman ideal not only one has  $\mu(I^2)= \ell(I)\mu(I)-{\ell(I)\choose 2}$,  but also  $\mu(I^{k})={\ell+k-2\choose k-1}\mu(I)-(k-1){\ell+k-2\choose k}$ for all $k\geq 1$. Indeed, this fact is a consequence of a result of B\"{o}r\"{o}czky, Santos and  Serra \cite[Corollary 7]{BSS}. Now this formula for the minimal  number of monomial   generators of $I^k$ of a Freiman ideal $I$ easily yields that $h_i=0$  for all $i\geq 2$, where $(1,h_1,h_2,\ldots)$ is the $h$-vector of the  fiber cone $F(I)$ of $I$. Then we apply  a result of Eisenbud and Goto  \cite[Corollary 4.5]{EG}  and  obtain that  $F(I)$ is normal. Since $F(I)$ is a toric  ring, Hochster's theorem \cite{Ho} finally yields that $F(I)$  is Cohen--Macaulay.

Even more surprising than the inequality  $\mu(I^2)\geq  \ell(I)\mu(I)-{\ell(I)\choose 2}$, and seemingly not noticed before in the study of toric rings, are  the following consequences of the above mentioned results about the $h$-vector of the fiber cone  $F(I)$ and of  Freiman's theorem,   which may phrased as follows: let $A$ be a standard graded toric $K$-algebra with $h$-vector $(1,h_1,h_2,\ldots)$. Then $h_2\geq 0$, and if $h_2=0$, then $h_i=0$ for all $i\geq 2$. The inequality $h_2\geq 0$ was first observed   in \cite[Corollary 2.6]{HMZ}.

Being a Freiman ideal is a very restrictive condition. Therefore, one can expect that for natural classes of monomial ideals arising in combinatorial context a nice classification of Freiman ideals is possible. We illustrate  this in two cases. In both cases we consider a finite simple graph $G$.  We call $G$ a Freiman graph, if its edge ideal $I(G)$   is  a Freiman ideal. This classification can be reduced to the case that $G$ is connected, see Corollary~\ref{connected}. In Theorem~\ref{hibi} it is shown that if $G$ is connected and  $H$ is the subgraph of $G$ whose  edges are the edges of all the  $4$-cycles of $G$, then $G$ is a Freiman graph  if and only if  there  exist no primitive even walks in $G$, or otherwise $H$ is bipartite of type $(2,s)$ for some integer $s$,  and there  exist no primitive even walks in $G$  of length $>4$. The proof is inspired by  the  classification of edge ideals with $2$-linear resolution, due to Ohsugi and Hibi \cite{OH1}. In the special case that $G$ is bipartite this classification is even more explicit. Indeed, in Corollary~\ref{bipartite}  it is shown, that if   $G$ is a connected bipartite graph,  then $G$ is Freiman if and only if $G$ is a tree,
or otherwise  there exists a bipartite subgraph $H$  of $G$ of type $(2,s)$ for some integer $s$, a subset $S$ of $V(H)$ and for each $w\in S$  an induced tree $T_w$ of $G$ with $V(H)\sect V(T_w)=\{w\}$.

The other case studied here is that of the matroidal ideal $I_M$ of the cycle matroid $M$ of $G$. We call a   matroid $M$ whose matroidal ideal $I_M$ is a Freiman ideal, a Freiman matroid. In Theorem~\ref{wehope} we show that  the cycle matroid of $G$ is Freiman if and only if $G$  contains at most one cycle. It is  a challenging open problem to classify all Freiman matroids. A few  results in this direction, even for polymatroids, can be found in \cite{HZ}.

As an interesting consequence of the classification of the cycle matroids which are Freiman, it is shown in Theorem~\ref{regularity} that if $r$ denotes  the regularity of the base ring of the cycle matroid of a graph $G$, then,  unless the base ring is a polynomial ring,  one has  $3\leq r\leq e$,
where $e$ is the number of the edges of  $G$.

\section{Freiman ideals and their fiber cones}

Let $K$ be a field and $S=K[x_1,\ldots,x_n]$ the polynomial ring in $n$ indeterminates over $K$, and  let $I\subset S$ be a monomial ideal. The unique minimal set of monomial generators of $I$ will be denoted by $G(I)$.   The number $\mu(I) =|G(I)|$ is  the minimal number of generators of $I$.  We write $\xb^{\cb}$ for the monomial $x_1^{c_1}\cdots x_n^{c_n}$ where $\cb=(c_1,\ldots,c_n)$.

Let $\ab=(a_1,\ldots,a_n)\in \mathbb{Z}^n$ be an integer vector with all $a_i>0$. We say that $I$ is {\em equigenerated (of degree $d$) with respect to $\ab$} or simply say that $I$ is {\em quasi-equigenerated}, if there exists an integer $d$ such that
\[
d=\langle \ab,\cb \rangle \quad \text{for all} \quad \xb^{\cb}\in G(I).
\]
Here $\langle \ab,\cb \rangle=\sum_{i=1}^n a_ic_i$ denotes the standard scalar product of $\ab$ and $\cb$. The ideal $I$ is  called {\em equigenerated} if $\ab=(1,1,\ldots,1)$.

The Krull dimension of the fiber cone $F(I)=\Dirsum_{k\geq 0}I^k/\mm I^k$ is called the {\em analytic spread} of $I$, and denoted $\ell(I)$. One always has
\[
\height I\leq \ell(I)\leq\min\{\mu(I),n\}.
\]

 Let $G(I)=\{u_1,\ldots,u_m\}$, and let $T=K[y_1,\ldots,y_m]$ be the polynomial ring over $K$ in the indeterminates $y_1,\ldots,y_m$.  Consider  the $K$-algebra homomorphism $\varphi\: T\to F(I)$ with $y_j\mapsto u_j+\mm I$ for $j=1,\ldots,m$. Then  $F(I)\iso T/J$, where  $J=\Ker \varphi$. Since $I$  is  quasi-equigenerated, it follows that $F(I)\iso K[u_1t,\ldots u_mt]\subset T[t]$, where $t$ is a new indeterminate over $T$.  Thus we see that $J$ is a toric prime ideal generated by binomials in $T$  which  are homogeneous with respect to the standard grading of $T$.

\medskip
The following theorem is a consequence of a famous theorem of Freiman \cite{F1} and its generalizations by  B\"or\"oczky
et al \cite{BSS}.

\begin{Theorem}
\label{freiman}
Let $I\subset S$ be a quasi-equigenerated monomial ideal whose  analytic
spread is  $\ell(I)$. Then
$$\mu(I^{k})\geq {\ell(I)+k-2\choose k-1}\mu(I)-(k-1){\ell(I)+k-2\choose k}$$ for all $k\geq 1$.
\end{Theorem}

\begin{proof}
Let $S(I)\subset \mathbb{Z}^n$ be the set of exponent vectors of the elements
of $G(I)$. Since $I$ is  a quasi-equigenerated monomial ideal,  $I^k$ are quasi-equigenerated monomial ideals for all $k\geq 1$. It follows that $\mu(I^k)=|S(I^k)|=|kS(I)|$ for all $k\geq 1$.
Let $d$ be the Freiman dimension of $S(I)$, then  from \cite[Corollary 2]{BSS},
\begin{eqnarray*}
|kS(I)|&\geq & {d+k-1\choose k-1}|S(I)|-(k-1){d+k-1\choose k}\\
&=& {d+k-1\choose k-1}\mu(I)-(k-1){d+k-1\choose k}.
\end{eqnarray*}

On the other hand, by \cite[Theorem 1.9]{HMZ}  one has $\ell(I)=d+1$. Thus the assertion follows.
\end{proof}

The above theorem implies in particular that $\mu(I^2)\geq l(I)\mu(I)-{l(I)\choose 2}$. In \cite{HZ}, the first and third author   of this paper called an equigenerated monomial ideal a Freiman ideal, if equality holds. Here we extend this definition and call a quasi-equigenerated ideal $I$ a {\em Freiman ideal}, if $$\mu(I^2)= l(I)\mu(I)-{l(I)\choose 2}.$$

\medskip
In the next theorem we will give some  characterizations of Freiman ideals  in terms of the Hilbert  series  of their  fiber cones.  To this end we recall some concepts from commutative algebra.

Let $R$ be a standard graded $K$-algebra. We denote its Krull dimension by $\dim R$ and its embedding dimension by  $\embdim R$. The Hilbert series $\Hilb_{R}(t)=\sum\limits_{k\geq 0}\dim_K R_kt^{k}$ of $R$ is  of the form
\[
\Hilb_{R}(t)=h(R;t)/(1-t)^{\dim R},
\]
where $h(R;t)=1+h_{1}t+h_{2}t^2+\cdots$ is a polynomial with $h_1=\embdim R-\dim R$. The polynomial $h(R;t)$  is called the $h$-polynomial of $R$,  and the finite coefficient vector of $h(R;t)$ is called the $h$-vector of $R$.

The multiplicity of $R$ is defined to be $e(R)=\sum_{i\geq 0}h_i$.
By Abhyankar \cite{A} it is known that
\[
\embdim R\leq e(R)+\dim R-1,
\]
if $R$ is a domain. The same inequality holds if $R$ is Cohen--Macaulay \cite{Sa}. The $K$-algebra  $R$   is said to have {\em minimal multiplicity} if
$\embdim(R)= e(R)+\dim R-1.$

\medskip
We  first observe

\begin{Corollary}
\label{remarkable}
Let $I\subset S$ be a quasi-equigenerated monomial ideal  with analytic
spread $\ell$, and let
  $(1,h_1,h_2,\ldots)$ be the $h$-vector of the fiber cone $F(I)$ of $I$. Then
\[
  \sum\limits_{i=2}^{k}{\ell+k-i-1\choose k-i}h_{i}\geq 0\quad \text{for all} \quad  k\geq 2.
\]
In particular, $h_2\geq 0$.
\end{Corollary}

\begin{proof} The Hilbert series of the fiber cone $F(I)$ of $I$ is
\begin{eqnarray*}
\Hilb_{F(I)}(t)&=&\sum\limits_{k\geq 0}\mu(I^{k})t^{k}=\frac{1+h_{1}t+h_{2}t^2+h_{3}t^3+h_{4}t^4+\cdots}{(1-t)^{\ell}}
\end{eqnarray*}
\begin{eqnarray*}
&=&(1+h_{1}t+h_{2}t^2+\cdots)(1+\ell t+{\ell+1\choose 2}t^2+{\ell+2\choose 3}t^3+\cdots)\\
&=&1+(h_{1}+\ell)t+(h_{1}\ell+{\ell+1\choose 2}+h_{2})t^2+\cdots\\
&+&[{\ell+k-1\choose k}+\sum\limits_{i=1}^{k}{\ell+k-i-1\choose k-i}h_{i}]t^k+\cdots.
\end{eqnarray*}
It follows  that
 $$\mu(I)=h_{1}+\ell,\ \mbox{ and}\  \ \mu(I^2)=h_{1}\ell+{\ell+1\choose 2}+h_{2}=\ell\mu(I)-{\ell\choose 2}+h_{2}.$$
Moreover,  for all $k\geq 3$ we have
\begin{eqnarray*}\mu(I^k)&=&{\ell+k-1\choose k}+{\ell+k-2\choose k-1}h_{1}+{\ell+k-3\choose k-2}h_{2}+\cdots+{\ell\choose 1}h_{k-1}+h_{k}\\
&=&{\ell+k-1\choose k}+{\ell+k-2\choose k-1}(\mu(I)-\ell)+\sum\limits_{i=2}^{k}{\ell+k-i-1\choose k-i}h_{i}   \\
&=&{\ell+k-2\choose k-1}\mu(I)-(k-1){\ell+k-2\choose k}+\sum\limits_{i=2}^{k}{\ell+k-i-1\choose k-i}h_{i}.
\end{eqnarray*}
Therefore, Theorem \ref{freiman} implies that $h_2\geq 0$ and
$\sum\limits_{i=2}^{k}{\ell+k-i-1\choose k-i}h_{i}\geq 0$ for all $k\geq 2$.
\end{proof}

Let  $I$ be a graded ideal.  An ideal  $J\subseteq I$  is  called  a {\it reduction} of  $I$ if $I^{k+1}=JI^{k}$
 for some nonnegative integer $k$.  The
{\it reduction number} of $I$ with respect to $J$ is defined to be   $$r_{J}(I)=\min\{k \mid I^{k+1}=JI^{k}\}.$$
A reduction $J$ of $I$ is called a {\em minimal reduction}  if  it does not properly contain any other  reduction of $I$.  If $I$ is equigenerated and $|K|=\infty$, then a graded  minimal reduction of $I$ exists. In this case
the {\it reduction number} of $I$ is defined to be the number  $$r(I)=\min\{ r_{J}(I)\mid J\ \mbox{ is a minimal reduction of}\ I\}.$$

\medskip
Now we are ready  to present the  main result of this section.

\begin{Theorem}
\label{characterization}
Let $I\subset S$ be a quasi-equigenerated monomial ideal with  analytic
spread  $\ell$,  and let $(1,h_1,h_2,\ldots)$ be the $h$-vector of the fiber cone $F(I)$ of $I$. Write $F(I)=T/J$, where $T$ is a polynomial ring over $K$, and $J$ is contained in the square of the graded maximal ideal of $T$.
Then the following conditions are equivalent:
\begin{enumerate}
\item[(a)] $I$ is a Freiman ideal.
\item[(b)]  $\mu(I^{k})={\ell+k-2\choose k-1}\mu(I)-(k-1){\ell+k-2\choose k}$ for all $k\geq 1$.
\item[(c)]  $\mu(I^{k})={\ell+k-2\choose k-1}\mu(I)-(k-1){\ell+k-2\choose k}$ for some $k\geq 2$.
\item[(d)] $h_2=0$.
\item[(e)] $h_i=0$ for all $i\geq 2$.
\item[(f)] $F(I)$ has minimal multiplicity.
\item[(g)] $F(I)$ is Cohen--Macaulay and the defining ideal  of  $F(I)$  has a $2$-linear free $T$-resolution.
\end{enumerate}
Moreover, if  $|K|=\infty$, then  the above conditions are equivalent to
\begin{enumerate}
\item[(h)] $F(I)$ is Cohen--Macaulay and $r(I)=1$.
\end{enumerate}
\end{Theorem}

\begin{proof} (a)\implies (b) follows from  \cite[Corollary 7]{BSS}, (b)\implies (c) is trivial,
and (c)\implies (b)  also follows from  \cite[Corollary 7]{BSS}.

In  the proof of Corollary \ref{remarkable}, we have seen that $\mu(I)=h_{1}+\ell, \mu(I^2)=\ell\mu(I)-{\ell\choose 2}+h_{2}$ and $$\mu(I^k)={\ell+k-2\choose k-1}\mu(I)-(k-1){\ell+k-2\choose k}+\sum\limits_{i=2}^{k}{\ell+k-i-1\choose k-i}h_{i}.$$ By definition,  $I$ is a Freiman ideal, if and only if   $\mu(I^2)= \ell\mu(I)-{\ell\choose 2}$. This is equivalent to saying that $h_2=0$. This proves (a)\iff(d).

(b)\iff (e) follows from \cite[Proposition 1.8 (b)]{HZ}.

(e)\implies (f):  Notice that  $e(R)=\sum_{i\geq 0}h_i$. If $h_i=0$ for $i\geq 2$, then
$$e(R)=h_0+h_1=1+h_1=1+\mu(I)-\ell.$$
This means that $F(I)$ has minimal multiplicity.

(f)\implies (g): After a base field extension, we may  assume that $K$ is algebraically closed. Since  $F(I)$ has minimal multiplicity, it follows that $F(I)$ is normal by \cite[Corollary 4.5]{EG}. Since $F(I)$ is a toric ring, Hochster's theorem \cite{Ho}  implies that
$F(I)$ is Cohen--Macaulay and $h_i=0$ for all $i\geq 2$.  Thus from \cite[Corollary 1.5]{HZ} we obtain that the defining ideal of  $F(I)$  has a $2$-linear free $T$-resolution.

(g)\implies (d): We may assume that $K$ is infinite. Since $F(I)$  is Cohen--Macaulay we can choose a regular sequence $t_1,\ldots,t_\ell$ of linear forms of $T_1$ which is also regular on $F(I)$. We let $\bar{T}=T/(t_1,\ldots,t_\ell)$ and denote by $\bar{J}$ the image of $J$ in $\bar{T}$. Since $t_1,\ldots,t_\ell$ is a regular sequence on $T/J$ and $J$ has a $2$-linear resolution, it follows that $\dim\bar{T}/\bar{J}=0$ and $\bar{J}$ has a $2$-linear resolution. It follows  that $\bar{J}=\mm_{\bar{T}}^2$, where $\mm_{\bar{T}}^2$ is the graded maximal ideal of $\bar{T}$, since the only $\mm_{\bar{T}}$-primary ideals with linear resolution are powers of the maximal ideal. Thus we see that $\overline{F(I)}_i=0$ for $i\geq 2$, where $\overline{F(I)}=F(I)/(y_1,\ldots,y_\ell)$,  and $y_i=t_i+J$ for $i=1,\ldots,\ell$. Since $y_1,\ldots,y_\ell$ is a regular sequence of linear form, the $h$-vector of $F(I)$ and that of $\overline{F(I)}$ coincide. Hence the conclusion follows.

(g)\implies (h): We assume that $K$ is infinite, and let $y_1,\ldots,y_\ell$ be as in the proof (g)\implies (d).
We write  $y_{i}=f_{i}+\mm I$  with $f_i\in I$ for $i=1,\ldots,\ell$, and let $J=(f_{1},\ldots, f_{\ell})$. Now $\overline{F(I)}_i=0$   for $i\geq 2$ is equivalent to saying that $I^{2}=JI+\frak{m}I^{2}.$ Thus,  Nakayama's lemma yields $I^{2}=JI$, as desired.

(h)\implies (d): Since $I$ is quasi-equigenerated,  $K$ is infinite and $F(I)$  is Cohen--Macaulay, a reduction ideal $J$ of $I$ with $r_J(I)=1$ can be chosen of the form $J=(f_1,\ldots,f_\ell)$ such that the sequence $y_i=f_i+\mm I$ with $i=1,\ldots, \ell$ is a regular sequence on $F(I)$. Since $I^2=JI$ it follows that $\overline{F(I)}_i=0$   for $i\geq 2$, where $\overline{F(I)}=F(I)/(y_1,\ldots,y_\ell)$. Since $h$-vector of $F(I)$ and that of $\overline{F(I)}$ coincide, the assertion follows.
\end{proof}

\section{Classes of Freiman ideals}

\subsection{Freiman graphs}

Let $K$ be a field and  $G$ be a finite simple graph on $[n]$.  The ideal $I(G)\subset S=K[x_1,\ldots,x_n]$ generated by the monomials $x_ix_j$ with $\{i,j\}\in E(G)$ is called the  {\em edge ideal } of $G$. We say that $G$ is a {\em Freiman graph}, if $I(G)$ is  a Freiman ideal. In this subsection we want to classify all Freiman graphs.  Notice that $F(I(G))$ is isomorphic to the edge ring $K[G]=K[x_ix_j\:\; \{i,j\}\in E(G)]$  of $G$.  Let $T=K[z_e\:\; e\in E(G)]$ be the polynomial ring in the indeterminates  $z_e$, and let $I_G$ be the kernel of the $K$-algebra homomorphism $T\to K[G]$ with $z_e\mapsto x_ix_j$ for $e=\{i,j\}$.  Then $K[G]\iso T/I_G$.

Certainly $G$ is Freiman, if $K[G]$ is the polynomial ring. The monomials $x_ix_j$ with $\{i,j\}\in E(G)$ are algebraic independent if and only if $K[G]$ is a polynomial ring, which is the case if and only if $I_G=0$.

In order to describe the generators of $I_G$ one has to consider walks in $G$.
A {\em walk} $W$ of length $r$ of $G$ is a sequence of vertices $i_0,i_1,\ldots,i_r$ such that $\{i_j,i_{j+1}\}$ is an edge of $G$ for $j=0,\ldots,r-1$. If  $i_0=i_r$, then $W$ is called a {\em closed walk}. The closed walk is called {\em even (odd)},  if $r$ is even (odd). It is called a {\em cycle}, if the vertices $i_j$ with $0\leq j\leq r-1$  are pairwise distinct.

\medskip
Now we have

\begin{Proposition}
\label{polynomialring}
The edge ring  $K[G]$   of $G$ is a polynomial ring,   if and only if each connected component of $G$ contains at most one cycle, and this cycle is odd.
\end{Proposition}

\begin{proof}
Let $G_1,\ldots, G_s$ be  the connected components of $G$. Then $K[G]=\bigotimes_{i=1}^sK[G_i]$. Thus $K[G]$ is a polynomial ring if and only if each $K[G_i]$ is a polynomial ring. Hence we may assume that $G$ is connected and have to show that $K[G]$ is a polynomial ring if and only if $G$ contains at most one cycle, and this cycle is odd.

Let $W: i_0,i_1,\ldots,i_{2s}$ be an even closed walk, and set $e_j=\{i_j,i_{j+1}\}$ for $j=0,\ldots,2s-1$.  Then $$f_W=\prod_{k=0}^{s-1}z_{e_{2k}}-\prod_{k=0}^{s-1}z_{e_{2k+1}}\in I_G,$$  and all binomial generators  of  $I_G$ are of this form, see \cite[Lemma 1.1]{OH1}. Thus $K[G]$ is a polynomial ring if and only if $f_W=0$ for each even closed walk in $G$.

Suppose that  $K[G]$  is a polynomial ring and $G$ contains an even cycle $C$. Then $f_C\neq 0$, a contradiction. Therefore,  $G$ contains no even cycles. Suppose $G$ contains  at least two odd cycles, say $C_1$ and $C_2$. Let $i$ be a vertex of $C_1$ and $j$ be  a vertex of $C_2$. Then, since $G$ is connected,  there exists a walk $W$ connecting $i$ with $j$. Now we consider the following even closed  walk $W'$: we first walk around $C_1$ starting with $i$.  Then we  walk along $W$ from $i$ to $j$.  Then we continue our walk around $C_2$ to come back to $j$. From there we walk back to $i$ along $W$,  but backwards. Obviously, $f_{W'}\neq0$. Thus there cannot be two odd cycles. This proves one direction of the proposition.

Conversely, suppose that $G$ contains only one odd cycle, it is enough to prove that $I_G=0$. By \cite[Lemma 3.1 and Lemma 3.2]{OH1},
the toric ideal $I_G$  is generated by the binomials $f_{W}$,
where $W$ is an  even closed walk of $G$ of  one of the following types:
(1) $W$ is an even cycle of $G$; (2)  $W=(C_1, C_2)$, where $C_1$ and $C_2$ are odd cycles of $G$ having exactly
one common vertex; (3) $W=(C_1,W_1,C_2,W_2)$, where $C_1$ and $C_2$ are odd cycles of $G$ having
no common vertex and where $W_1$ and $W_2$ are walks of $G$ both of which combine a vertex $v_1$ of $C_1$ and a vertex $v_2$ of $C_2$.
Walks of these three types are called primitive.

By the hypothesis that  $G$ contains only one odd cycle, there cannot exist primitive  even closed walks of $G$. Therefore,  $I_G=0$.
\end{proof}

\begin{Corollary}
\label{connected}
Let $G$ be a finite  simple graph with $r$ connected components. Then following conditions are equivalent:
\begin{enumerate}
\item[(a)] $G$ is a Freiman graph.

\item[(b)] All connected components of $G$  are  Freiman graphs  and at least $r-1$ of its  connected components  contain at most one cycle, and this cycle is odd.
\end{enumerate}
\end{Corollary}

\begin{proof}
Let $G_1,\ldots,G_r$ be the connected components of $G$, and let  $\Hilb_{K[G_i]}(t)=h(K[G_i];t)/(1-t)^{d_i}$ with $d_i=\dim K[G_i]$.  Then, since $K[G]=\bigotimes_{i=1}^rK[G_i]$, it follows that
\[
\Hilb_{K[G]}(t)=\prod_{i=1}^r\Hilb_{K[G_i]}(t)=h(K[G];t)/(1-t)^d,
\]
where $h(K[G];t)=\prod_{i=1}^rh(K[G_i];t)$ and $d=\sum_{i=1}^rd_i=\dim K[G]$.

(a)\implies (b): Since $G$ is Freiman, Theorem~\ref{characterization} implies that  $\deg h(K[G];t)\leq 1$.  Therefore, $\deg h(K[G_i]; t)\leq 1$ with equality for at most one $i$. Thus (b) follows from Proposition~\ref{polynomialring} together with  Theorem~\ref{characterization}.

(b)\implies (a): The assumptions in (b) together with Theorem~\ref{characterization}  and Proposition~\ref{polynomialring}  imply that $\deg h(K[G_i];t)\leq 1$ and  $\deg h(K[G_i];t)=0$ for all $h(K[G_i];t)$ but possibly one of them. Thus $\deg h(K[G];t)\leq 1$, which by Theorem~\ref{characterization} implies that $G$ is Freiman.
\end{proof}

\medskip
Due to Corollary~\ref{connected} it is enough to consider only connected graphs for the classification of Freiman graphs.

\begin{Theorem}
\label{hibi}
Let $G$ be a finite simple  connected graph, and let $H$ be the subgraph of $G$ whose  edges are the edges of all the  $4$-cycles of $G$. Then  $G$ is Freiman, if and only if  there  exist no primitive even walks in $G$, or otherwise $H$ is bipartite of type $(2,s)$
for some integer $s$,  and there  exist no primitive even walks in $G$  of length~$>4$.
\end{Theorem}

\begin{proof} If $G$ has no primitive even walks, then $K[G]$ is the polynomial ring, and hence $G$ is Freiman.
Suppose now that   $H$ is bipartite of type $(2,s)$ for some $s$, and there  exist no primitive even walks in $G$ of length~$>4$.
Since  $G$ admits  no primitive even walks of length~$>4$, the ideal $I_G$ is generated by the binomials $f_C$ with $C$ a $4$-cycle of $G$. In particular, $I_G=I_H$, and $K[G]$ is a polynomial extension of $K[H]$. Since $H$ is  bipartite of type  $(2,s)$, the ideal $I_H$ may be viewed as the ideal of $2$-minors of a  $2\times s$-matrix of indeterminates. Hence the Eagon-Northcott complex \cite{EN}  provides a $2$-linear resolution of $I_H(=I_G)$. By Theorem~\ref{characterization}, this implies that $G$ is Freiman.

Conversely, suppose that $G$ is Freiman. Then Theorem~\ref{characterization} implies that $I_G$ has no generators of degree $>2$, and hence  $G$ admits  no primitive even walks of length~$>4$.  We may assume that $H$ is bipartite of type $(r,s)$ with $r\leq s$. Since the edge set of $H$ is the union of the edges of the  $4$-cycles of $G$, we must have $r\geq 2$. Also it follows from this fact that $H$ is a complete bipartite graph if $r=2$. Indeed, suppose that one edge is missing, say the edge $\{2,r\}$. Then $\{1,r\}\in E(H)$ and $\{1,r\}$ belongs to a $4$-cycle. Any $4$-cycle containing the edge $\{1,r\}$  is of the form $r,1,i,2,r$ for some $i\neq r$. Hence  $\{2,r\}\in E(H)$, a contradiction.

It remains to be show that   $r=2$. Suppose $r>2$, and let $\{v_1,\ldots,v_r\}\union\{w_1,\ldots,w_s\}$ be the bipartition of the graph $H$. Since $H$ is the union of the edges of the $4$-cycles of $G$, we may assume that,  after a relabeling of the vertices,  $C_1:v_1,w_1,v_2,w_2,v_1$  is a $4$-cycle of $H$. Since $r>2$, there  must exist  another $4$-cycle $C_2$ of $H$ which contains the vertex  $v_3$. Then  $C_1$ and $C_2$ have only one  vertex or one  edge in common,  or the cycles $C_1$ and $C_2$  are disjoint. But in the last case there is a walk   in $G$ connecting  a vertex of $C_1$ with a vertex of $C_2$ since $G$ is connected.  In any case we have in this situation a primitive even walk of length $>4$, a contradiction.
\end{proof}

\begin{Corollary}
\label{bipartite}
Let $G$ be a finite  connected bipartite graph.  Then $G$ is Freiman if and only if $G$ is a tree,
or there exists a bipartite subgraph $H$  of $G$ of type $(2,s)$, a subset $S$ of $V(H)$ and for each $w\in S$  an induced tree $T_w$ of $G$ with $V(H)\sect V(T_w)=\{w\}$.
\end{Corollary}

\begin{proof}
If $G$ is a tree
or if there exists a bipartite subgraph $H$  of $G$ satisfying the conditions of the corollary, then the subgraph $H$ is the union of all $4$-cycles of $G$, and moreover there exist no primitive even walks in $G$  of length~$>4$. Therefore, $G$ is Freiman, by Theorem~\ref{hibi}.

Conversely, assume that $G$ is Freiman and let $H$ be as defined in Theorem~\ref{hibi}. Then  Theorem~\ref{hibi} implies that $H=\emptyset$  or $H$ is bipartite of type $(2,s)$  for some $s$. If $H=\emptyset$, then Theorem~\ref{hibi} implies that $G$  has no even cycles of any length. Since $G$ is bipartite and connected, this implies that $G$ is a tree.

Suppose now that  $H$ is bipartite of type $(2,s)$. If $G=H$, then there is nothing to prove. Otherwise there exists $v\in V(G)\setminus V(H)$. Since $G$ is connected there exists a walk  $P_1$ outside of $H$ connecting $v$ with $w_1\in V(H)$. Assume there is another walk  $P_2$ outside $H$ connecting $v$ with $w_2\in V(H)$  with $w_1\neq w_2$. Then there exists  $w$ belonging to $P_1$ and $P_2$ such that the subwalk  of $P_1$ connecting $w$ with $w_1$ and the subwalk  of $P_2$ connecting $w$ with $w_2$  have only the vertex $w$ in common. Then we get a cycle walking from $w_1$ to $w$ to $w_2$,  and from $w_2$ back to $w_1$ (inside $H$). Since $G$ is bipartite, this cycle must be even, a contradiction because by Theorem~\ref{hibi}, $G$ has even cycles of  length   at most $4$. But all even cycles of length $4$ belong to $H$.  It  follows that each vertex $v$ of $G$ not belonging to $H$ is connected with at most one vertex $w\in H$.

Now let $w\in H$,  and let $V$ be the set of vertices $v\in V(G)\setminus V(H)$ which are connected to $w$ via a walk. We denote by $T_w$ the restriction  of $G$ to $V$.  As shown above,  $V(T_{w_1})\sect V(T_{w_2})=\emptyset$.    Since $G$ has no even closed walks of length $>4$,  it follows that there is for each $v\in V(T_w)$  a unique walk,  whose edges are all distinct from one another, and which  connects $v$ with $w$. In fact, if there exist two such  walks connecting $v$ with $w$, say $P_1$, $P_2$, then  their lengths  must be both  odd or  both even,  since $G$ is bipartite. By Theorem \ref{characterization}, the lengths of $P_1$ and $P_2$ are both $2$. In fact, since $G$ is Freiman, there exist no primitive even walks in $G$  of length~$>4$, thus the lengths of $P_1$ and $P_2$ are both $1$ or $2$. If their  lengths are  both $1$, then $G$ have a multiple edge connecting $v$ with $w$. But the graph $G$ is simple without multiple edge, thus  the union of $P_1$ and $P_2$ is a $4$-cycle of $G$. This implies $v\in V(H)$, a contradiction. Now the property that for each $v\in V(T_w)$ there exists a unique walk connecting $v$ with $w$, implies that $T_w$ is a tree.
\end{proof}

A typical example of a bipartite Freiman graph is shown in the next figure.

\begin{center}
\psset{xunit=0.7cm,yunit=0.7cm,algebraic=true,dimen=middle,dotstyle=o,dotsize=5pt 0,linewidth=1.6pt,arrowsize=3pt 2,arrowinset=0.25}
\begin{pspicture}(2.,-5.)(11.,5.)
\psline[linewidth=1.pt](5.44,1.62)(4.28,-0.04)
\psline[linewidth=1.pt](5.44,1.62)(6.24,-0.04)
\psline[linewidth=1.pt](7.22,1.58)(6.24,-0.04)
\psline[linewidth=1.pt](7.22,1.58)(8.18,-0.04)
\psline[linewidth=1.pt](5.44,1.62)(8.18,-0.04)
\psline[linewidth=1.pt](7.22,1.58)(4.28,-0.04)
\psline[linewidth=1.pt](7.22,1.58)(8.28,2.26)
\psline[linewidth=1.pt](8.28,2.26)(7.96,3.82)
\psline[linewidth=1.pt](8.28,2.26)(9.8,2.12)
\psline[linewidth=1.pt](5.44,1.62)(6.3,2.76)
\psline[linewidth=1.pt](5.44,1.62)(3.76,1.92)
\psline[linewidth=1.pt](6.24,-0.04)(5.1,-1.1)
\psline[linewidth=1.pt](5.1,-1.1)(6.04,-2.18)
\psline[linewidth=1.pt](5.1,-1.1)(3.74,-1.9)
\psline[linewidth=1.pt](6.04,-2.18)(5.24,-3.5)
\psline[linewidth=1.pt](6.04,-2.18)(7.4,-2.96)
\begin{scriptsize}
\psdots[dotstyle=*,linecolor=blue](5.44,1.62)
\psdots[dotstyle=*,linecolor=blue](4.28,-0.04)
\psdots[dotstyle=*,linecolor=blue](6.24,-0.04)
\psdots[dotstyle=*,linecolor=blue](7.22,1.58)
\psdots[dotstyle=*,linecolor=blue](8.18,-0.04)
\psdots[dotstyle=*,linecolor=blue](8.28,2.26)
\psdots[dotstyle=*,linecolor=blue](7.96,3.82)
\psdots[dotstyle=*,linecolor=blue](9.8,2.12)
\psdots[dotstyle=*,linecolor=blue](6.3,2.76)
\psdots[dotstyle=*,linecolor=blue](3.76,1.92)
\psdots[dotstyle=*,linecolor=blue](5.1,-1.1)
\psdots[dotstyle=*,linecolor=blue](6.04,-2.18)
\psdots[dotstyle=*,linecolor=blue](3.74,-1.9)
\psdots[dotstyle=*,linecolor=blue](5.24,-3.5)
\psdots[dotstyle=*,linecolor=blue](7.4,-2.96)
\end{scriptsize}
\end{pspicture}

\vspace{2mm}
Figure \ 1
\end{center}

\medskip

\subsection{Cycle  matroids} Given a finite  simple graph $G$ on $[n]$.  The {\em cycle matroid} of  $G$  is the
matroid whose ground set  is  $E(G)$,   and whose bases are   the sets $E(F)$ with $F$ a spanning forest of $G$. A {\em spanning forest} of $G$  is a maximal acyclic  subgraph  of $G$.

A matroid $M$ is called  {\em graphic},  if for some simple graph $G$, $M$  is isomorphic to the cycle graph of $G$.

Let $K$ be   a field. Let  $M$ be a matroid on the ground set $[m]$, and $\B$ the  set of its bases. To $M$  we attach the ideal $I_M\subset S=K[x_1,\ldots,x_m]$,  whose generators are the monomials $u_B$ with $u_B=\prod_{i\in B}x_i$ and $B\in \B$. The ideal $I_M$ is called the {\em matroidal ideal} of $M$. The fiber cone of $I_M$ is the base ring of $M$. We say that $M$ is a {\em Freiman matroid}, if  $I_M$ is a Freiman ideal.

Let $E(G)=\{e_1,\ldots,e_m\}$, and $M$ the cycle matroid of $G$.   Then  $I_M\subset K[x_1,\ldots,x_m]$  is the ideal generated by the monomials $u_F=\prod_{e_i\in F}x_i$ with $F$ a spanning forest of~$G$.

\medskip
As the main result of this subsection we have

\begin{Theorem}
\label{wehope}
Let $M$ be the cycle matroid of $G$. Then the following conditions are equivalent:
\begin{enumerate}
\item[(a)] $M$ is a Freiman matroid.
\item[(b)]  $G$ contains at most one cycle.
\item[(c)] $F(I_M)$ is a polynomial ring.
\end{enumerate}
\end{Theorem}

We will need the following lemmas  before proving Theorem~\ref{wehope}. First,  we have

\begin{Lemma}
\label{pure}
Let $G$ be a  finite simple  graph whose cycle matroid is $M$. Let $E$ be a subset of the edges of $G$, $\Gamma$  the graph with $E(\Gamma)=E$ and $N$ be the cycle matroid of $\Gamma$. Then $N$ is a Freiman matroid.
\end{Lemma}

\begin{proof}
We observe that $F(I_N)$ is a combinatorial pure subring of $F(I_M)$ in the sense of \cite{OHH}. We can write $F(I_M)=B/J$ and $F(I_N)=A/I$, where $A$ and $B$ are polynomial rings and  $J\subset B$ and $I\subset A$  are  the binomial relation ideals of $F(I_M)$ and $F(I_N)$, respectively. By \cite[Corollary 2.5]{OHH} we have $\beta_{ij}^A(I)\leq \beta_{ij}^B(J)$  for all $i,j$. Since $M$ is a Freiman matroid, it follows from Theorem~\ref{characterization} that $J$ has a $2$-linear $B$-resolution. Therefore the above inequality of Betti numbers implies  that $I$ has a $2$-linear $A$-resolution. Furthermore,  since $F(I_N)$ is normal, $F(I_N)$ is Cohen--Macaulay,  by Hochster's theorem, see  \cite{Ho}. We now apply again Theorem~\ref{characterization} and conclude that $N$ is a Freiman matroid.
\end{proof}

\begin{Lemma}
\label{vertexincommon}
Let $G$ be a  finite simple  graph consisting of two cycles $C_1$ and $C_2$ which have at most  one vertex in common.  Then the cycle matroid of $G$ is not a Freiman matroid.
\end{Lemma}

\begin{proof}
Let $M_i$ be the cycle matroid of $C_i$, and let $r_i=|E(C_i)|$.  Then $I_{M_i}$ can be identified with the squarefree Veronese ideal $I_{r_i,r_{i}-1}$ of degree $r_i-1$ in $r_i$ indeterminates. Hence the fiber cone of $I_{M_i}$ is isomorphic   to   $K[I_{r_i,r_i-1}]$ and
this ring is isomorphic to a polynomial ring $S_i$ over $K$ in $r_i$ indeterminates. Since  $C_1$ and $C_2$  have at most  one vertex in common,  each spanning forest  of $G$  is obtained from $G$ by removing one edge from $C_1$ and another  edge from $C_2$. This shows that $I_M=I_{M_1}I_{M_2}$. Hence the fiber cone of $I_M$ is isomorphic to the Segre product $T=S_1\# S_2$ of the polynomial rings $S_1$ and $S_2$.

By \cite[Theorem 4.2.3]{GW} we know that $\dim T=r_1+r_2-1$. Thus $\Hilb_T(t)=h(T;t)/(1-t)^{r_1+r_2-1}$, and so   the  $a$-invariant $a(T)$ of $T$ is equal to $\deg h(T;t)-(r_1+r_2-1)$ (see \cite[Definition 4.4.4]{BH}). Hence   we get
\begin{eqnarray}
\label{a}
\deg h(T;t)(t)=a(T)+(r_1+r_2-1).
\end{eqnarray}

On the other hand, if $\omega_T$ is the canonical module of $T$, then
\begin{eqnarray}
\label{b}
a(T)=-\min\{i\:\; (\omega_T)_i\neq 0\},
\end{eqnarray}
see \cite[Definition 3.1.4]{GW}. We now use another result of Goto and Watanabe \cite[Theorem 4.2.3]{GW}  which says that $\omega_T\iso \omega_{S_1}\# \omega_{S_2}$. Since $\omega_{S_i}=S_i(-r_i)$, we have  $\omega_T\iso  S_1(-r_1)\#S_2(-r_2)$. The $i$th graded component $(\omega_T)_i$ of $\omega_T$ is isomorphic to
\[
(S_1(-r_1)\# S_2(-r_2))_i=(S_1)_{i-r_1}\tensor (S_2)_{i-r_2}.
\]
Therefore, $\min\{i\:\; (\omega_T)_i\neq 0\}=\max\{r_1,r_2\}$. Combining (\ref{a}) with (\ref{b}),  we get
\[
\deg h(T;t)=r_1+r_2-1-\max\{r_1,r_2\}.
\]
Since each cycle has at least 3 edges, we see that $r_i\geq 3$. It follows that $\deg h(T;t)\geq 2$. Thus Theorem~\ref{characterization} implies that $M$ is not a Freiman matroid.

In the case that $C_1$ and $C_2$ are connected by a walk $P$, we have  $I_M=uI_{M_1}I_{M_2}$, where $u$ is the monomial whose factors correspond to the edge of $P$. Then $K[I_M]=K[uI_{M_1}I_{M_2}]\iso K[I_{M_1}I_{M_2}]$, and the rest of the proof is as before.
\end{proof}

\begin{Lemma}
\label{pathincommon}
Let $G$ be a finite simple graph consisting of a cycle  $C$ of length $r$ and  a walk  $W$  of length $s$ which connects two non-consecutive  vertices of $C$. Then the cycle matroid $M$ of $G$ is not a Freiman matroid.
\end{Lemma}

\begin{proof} By hypothesis, the walk  $W$ divides cycle $C$ into two walks   $W_1$ and  $W_2$, where $W_1$, $W_2$ and $W$ have  common starting point and endpoint. Thus $C=W_1\cup W_2$.  For  simplicity we set $W_3=W$.  Let $M_i$ be the cycle matroid of $W_i$, $r_i=|E(W_i)|$, and  let $M_4$ be the cycle matroid of $C$ and  $r_4=|E(C)|$.  Then $r_3=s$, $r_4=r_1+r_2=r$. By the proof   of Lemma \ref{vertexincommon}, the fiber cone  $F(I_{M_i})$ of $I_{M_i}$ is isomorphic  to $K[I_{r_i,r_i-1}]$ and this ring is isomorphic to a polynomial ring $S_i$ over $K$ in $r_i$ indeterminates.  Therefore, we get $I_M=I_{r_3,r_{3}-1}I_{{r_1+r_2,r_{1}+r_{2}-1}}+JI_{r_3,r_3}$,  where $J=I_{r_1,r_1-1}I_{r_2,r_2-1}$. We have  $I_{r_3,r_3}\subset I_{r_3,r_{3}-1}$ and $I_{{r_1+r_2,r_{1}+r_{2}-1}}\subset J$, Let  $S=S_3\otimes S_4$ and $\Hilb_{F(I_{M})}(t)=h(S;t)/(1-t)^{\ell}$, where $\ell=\dim(F(I_{M}))$. By \cite[Theorem 1.6]{HT}, we have

  \begin{eqnarray*}
 \Hilb_{S/I_{M}}(t)&=&\Hilb_{S/(I_{r_3,r_{3}-1}I_{{r_1+r_2,r_{1}+r_{2}-1}}+JI_{r_3,r_3})}(t)\\
 &=&\Hilb_{S_3/I_{r_3,r_{3}-1}}(t)\Hilb_{I_{{r_1+r_2,r_{1}+r_{2}-1}}}(t)+\Hilb_{S_4/J}(t)\Hilb_{I_{r_3,r_3}}(t)\\
 &+&\Hilb_{S_3/I_{r_3,r_3}}(t)\Hilb_{S_4/I_{{r_1+r_2,r_{1}+r_{2}-1}}}(t)\\
 &=&\Hilb_{I_{{r_1+r_2,r_{1}+r_{2}-1}}}(t)+\Hilb_{S_4/J}(t)\Hilb_{I_{r_3,r_3}}(t)+\Hilb_{S_3/I_{r_3,r_3}}(t).
 \end{eqnarray*}
Hence
\begin{eqnarray*}
\Hilb_{F(I_{M})}(t)&=&\frac{1}{(1-t)^{r_3+r_4}}-\Hilb_{S/I_{M}}(t)\\
&=&\frac{1}{(1-t)^{r_3+r_4}}-[\sum\limits_{m=0}^{\infty}{m+r_4-1\choose m}t^m+(\frac{1}{(1-t)^{r_4}}-\frac{h(S_4;t)}{(1-t)^{r_4-1}})\sum\limits_{m=0}^{\infty}t^m\\
&+&\frac{1}{(1-t)^{r_3-1}}]\\
&=&\frac{1}{(1-t)^{r_3+r_4}}-[\frac{1}{(1-t)^{r_4}}+[\frac{1}{(1-t)^{r_4}}-\frac{h(S_4;t)}{(1-t)^{r_4-1}}]\frac{1}{1-t}
+\frac{1}{(1-t)^{r_3-1}}]\\
&=&\frac{1}{(1-t)^{r_3+r_4}}-\frac{(1-t)^{r_3}+(1-t)^{r_3-1}[1-(1-t)h(S_4;t)]+(1-t)^{r_4+1}}{(1-t)^{r_3+r_4}}
\\
&=&\frac{1-(1-t)^{r_3}-(1-t)^{r_3-1}[1-(1-t)h(S_4;t)]-(1-t)^{r_4+1}}{(1-t)^{r_3+r_4}}\\
&=&\frac{1-(1-t)^{s}-(1-t)^{s-1}[1-(1-t)h(S_4;t)]-(1-t)^{r+1}}{(1-t)^{r+s}},
\end{eqnarray*}
where $\deg h(S_4;t)=\min\{r_1,r_2\}-1$, as seen in  the proof of Lemma \ref{vertexincommon}.

\medskip
We consider the following two cases:

(1)  If $s=1$, then $\Hilb_{F(I_{M})}(t)=\frac{h(S_4;t)-1-(1-t)^{r}}{(1-t)^{r}}$ and $h(S;t)=h(S_4;t)-1-(1-t)^{r}$, Since $\deg\,h(S_4;t)=\min\{r_1,r_2\}-1$ and $r=r_1+r_2$,
 $\deg\,h(S;t)\geq \max\{r_1,r_2\}+1$. Notice that $r_i\geq 2$, we have that $\deg\,h(S;t)\geq 3$;

(2)  If $s\geq 2$, then  $h(S;t)=1-(1-t)^{s}-(1-t)^{s-1}[1-(1-t)h(S_4;t)]-(1-t)^{r+1}$ such that
 $h(S;1)=1$ and $\deg\,h(S;t)\geq 5$.

Thus Theorem~\ref{characterization} implies that the cycle matroid $M$ of $G$ is not a Freiman matroid.
\end{proof}

\medskip
Now we are ready to prove   Theorem \ref{wehope}.

\begin{proof}
(a)\implies (b): Suppose there is a cycle $C$  for which there exists a walk  $W$    which connects two non-consecutive vertices of $C$. Then let $\Gamma$ be the graph with edges $E(W)\union E(C)$. By Lemma~\ref{pathincommon} the cycle matroid of $\Gamma$ is not a Freiman matroid, and so $M$ is not a Freiman matroid by Lemma~\ref{pure}, a contradiction. Thus there is  no cycle $C$ of $G$  such that there exists a walk  $W$    which connects two non-consecutive vertices of $C$.

Let us now assume that $G$ contains at least two cycles $C_1$ and $C_2$. We discuss two cases: (i) $V(C_1)\sect V(C_2)=\emptyset$, and  (ii) $V(C_1)\sect V(C_2)\neq \emptyset$.

Let $\Gamma$ be the graph with $E(\Gamma)= E(C_1)\union E(C_2)$.
In case (i), the cycle matroid of $\Gamma$ is not a Freiman matroid, see Lemma~\ref{vertexincommon}. Therefore, according to Lemma~\ref{pure},  also $M$ is not a Freiman matroid.

In case (ii), by  using Lemma~\ref{pure},  we may assume that $\Gamma=G$. Let $S=V(C_1)\sect V(C_2)$, and $\Sigma$ the graph  which is obtained from $G$ by restriction to $S$. Then $\Sigma$ is a disjoint union of walks and   some isolated vertices. If $\Sigma$ is just  a  single vertex, $M$ is not  a Freiman matroid  by  Lemma~\ref{vertexincommon},  and if $\Sigma$ is just a single walk, then the edges $E(C_1)\union E(C_2)\setminus E(\Sigma)$ form a cycle $C$ and $\Sigma$ is a walk   which connects two non-consecutive vertices of $C$. Thus  Lemma~\ref{pathincommon} implies that $M$ is not a Freiman matroid. Finally assume that $\Sigma$ has more than one  connected component. Let $V(C_1)=\{v_1,\ldots,v_r\}$ and $V(C_2)=\{w_1,\ldots, w_s\}$ with a counter-clockwise labeling of the vertices. We may assume without loss of generality that $v_1=w_1$ and $v_2\neq w_2$, since $\Sigma$ is not connected. There exist integers $a,b$ with the property that $v_a=w_b$ and $v_i\neq w_j$ for all $i,j$ with $1<i<a$ and $1<j<b$.  It then follows that  the  walk $W'$ with $V(W')=\{w_1,w_2,\ldots, w_b\}$ is a walk connecting the vertices  $v_1$ and $v_a$ of $C_1$. Thus $M$ is not a Freiman matroid.

(b) \implies (c): Let $G_1,\ldots,G_s$ be the connected components of $G$ and let $M_i$ be the cycle matroid of $G_i$.  Since $G$ contains at most one cycle, we may assume that the components $G_i$ are trees for $i\geq 2$.   Then,  as in the proof of Lemma~\ref{pathincommon},  we see that $I_{M_1}=uI_{r,r-1}$ if $G_1$ contains a cycle $C$ with $|E(C)|=r$, or otherwise  $I_{M_1}$ is a principal  monomial ideal. Moreover,  all  $I_{M_i}$ for $i\geq 2$ are  principal  monomial ideals.  Since  $I_M=I_{M_1}\cdots I_{M_s}$, we have  $K[I_M]\iso K[I_{M_1}]$,  and  hence $K[I_M]$ is a polynomial ring.

(c)\implies(a) is obvious.
\end{proof}

We close this paper with discussing the regularity of the base ring of the cycle matroid of a simple graph.  Let, as before, $G$ be a finite simple graph on $[n]$ and $E(G) = \{e_1, \ldots, e_m \}$ its edge set.  Let ${\bf e}_1, \ldots, {\bf e}_m$ denote the canonical coordinate vectors of $\RR^m$.  Given a subset $F \subset E(G)$, one defines $\rho(F) = \sum_{e_i \in F} {\bf e}_i \in \RR^m$.  For example, $\rho(\{ e_2, e_3, e_5 \}) = {\bf e}_2 + {\bf e}_3 + {\bf e}_5 = (0,1,1,0,1,0,\ldots,0) \in \RR^m$.  The {\em base polytope} $\P(M)$ of the cycle matroid $M$ of $G$ is the convex polytope which is the convex hull of the finite set $\{ \rho(F) : \rho(F) \in \B \}$ in $\RR^m$.  One has $\dim F(I_M) = \dim \P(M) + 1$.

Recall that, when $G$ is connected, a vertex $i \in [n]$ of $G$ is said to be a {\em cut vertex} if the induced subgraph $G|_{[n]\setminus\{i\}}$ of $G$ on $[n]\setminus\{i\}$ is disconnected.  A finite connected simple graph $G$ is called {\em $2$-connected} if no vertex of $G$ is a cut vertex of $G$.  It then follows from \cite[Theorem 5]{Whi} that

\begin{Lemma}
\label{basedim}
Let $G$ be a finite $2$-connected simple graph on $[n]$ with $e$ edges.  Let $M$ be the cycle matroid of $G$  and $\P(M)$ the base polytope of $M$.  Then
\[
\ell(I_M) = \dim \P(M) + 1 = e.
\]
\end{Lemma}

\begin{Corollary}
\label{basedimCOR}
Let $G$ be a finite connected simple graph on $[n]$ with $e$ edges and write $c$ for the number of cut vertices of $G$.  Let $M$ be the cycle matroid of $G$ and $\P(M)$ the base polytope of $M$.  Then
\[
\ell(I_M) = \dim \P(M) + 1 = e - c.
\]
\end{Corollary}

\begin{proof}
Let, say, $n \in [n]$ be a cut vertex of $G$.  Let $G'$ and $G''$ be the connected components of the induced subgraph $G|_{[n-1]}$.  Write $U$ for the vertex set of $G'$ and $V$ that of $G''$.  Let $M'$ be the cycle matroid of $G|_{U \cup \{n\}}$ and $M''$ that of $G|_{V \cup \{n\}}$.  Let $e'$ be the number of edges of $G|_{U \cup \{n\}}$ and $e''$ that of $G|_{V \cup \{n\}}$.  Since the base ring $F(I_M)$ is the Segre product of $F(I_{M'})$ and $F(I_{M''})$, it follows that $\dim F(I_M) = \dim F(I_{M'}) + \dim F(I_{M''}) - 1$.  Let $c'$ be the number of cut vertices of $G'$ and $c''$ that of $G''$.  Then using induction yields $\dim F(I_{M'}) = e' - c'$ and $\dim F(I_{M''}) = e'' - c''$.  Hence $\dim F(I_M) = (e' - c') + (e'' - c'') - 1$.  Since $e=e'+e''$ and $c=c'+c''+1$, the desired formula $\dim F(I_M) = \dim \P(M) + 1 = e - c$ follows.
\end{proof}

\begin{Corollary}
\label{basedimCORCOR}
Let $G$ be a finite disconnected simple graph on $[n]$ with $e$ edges and $G_1, \ldots, G_s$ its connected component.  Let $c_i$ denote the number of cut vertices of $G_i$ for $1 \leq i \leq s$.  Let $M$ be the cycle matroid of $G$  and $\P(M)$ the base polytope of $M$.  Then
\[
\ell(I_M) = \dim \P(M) + 1 = e - c - s + 1.
\]
\end{Corollary}

\begin{proof}
Let $M_i$ be the cycle matroid of $G_i$ for $1 \leq i \leq s$.
Since the base ring $F(I_M)$ is the Segre product of the base rings $F(I_{M_1}), \ldots, F(I_{M_s})$, the required formula follows immediately from Corollary \ref{basedimCOR}.
\end{proof}

Let $G$ be a finite $2$-connected simple graph on $[n]$ with $e$ edges and $\P(M)$ the base polytope of the cycle matroid $M$ of $G$.  Thus $\dim \P(M) = e - 1$.  We introduce the sequence of integers $(\delta_0, \delta_1, \delta_2, \ldots)$ by the formula
\[
(1 - \lambda)^e \sum_{i=0}^{\infty}\delta_i t^i = \sum_{i=0}^{\infty} |i\P(M) \cap \ZZ^m|\,t^i,
\]
where $i\P(M)$ is the $i$th dilated polytope $\{ i\alpha : \alpha \in \P(M) \}$ of $\P(M)$.  It then follows that $\delta_i = 0$ for $i > \dim \P(M) = e - 1$.  We say that $\delta(\P(M))=(\delta_0, \delta_1, \ldots, \delta_{e-1})$ is the {\em $\delta$-vector} of $\P(M)$ and $\delta(\P(M);t) = \sum_{i=0}^{e-1} \delta_i t^i$ is the {\em $\delta$-polynomial} of $\P(M)$.  Furthermore, introducing the integer $r_0 \geq 1$ by
\[
r_0 = \min\{  r\in \ZZ : r\geq 1 \ \text{and}\ r(\P(M) \setminus \partial \P(M)) \cap \ZZ^m \neq \emptyset \},
\]
where $\P(M) \setminus \partial \P(M)$ is the relative interior of $\P(M)$, one has the formula
\begin{eqnarray}
\label{degdelta}
\deg \delta(\P(M);t) = e - r_0.
\end{eqnarray}
We refer the reader to \cite{HibiRedBook} for the detailed information about $\delta$-vectors.

Since the base polytope $\P(M)$ possesses the {\em integer decomposition property} (\cite[p.~250]{HH1}), it follows that

\begin{Lemma}
\label{idp}
The $\delta$-polynomial $\delta(\P(M);t)$ of $\P(M)$ coincides with the $h$-polynomial of $F(I_M)$.
\end{Lemma}

Since the base ring $F(I_M)$ is Cohen--Macaulay, it is well-known  that the regularity $\reg(F(I_M))$ of $F(I_M)$ is one more than the degree of the $h$-polynomial $h(F(I_M);t)$ of $F(I_M)$.  In other words, the regularity $\reg(F(I_M))$ is one more than the degree of the $\delta$-polynomial $\delta(\P(M);t)$ of $\P(M)$.  Since the hyperplanes $x_i = 0$ and $x_i = 1$ of $\RR^m$ are supporting hyperplanes of $\P(M)$ for $1 \leq i \leq n$, it follows that no lattice point of $\RR^m$ belongs to the relative interior of $\P(M)$.  Hence,  by using the formula (\ref{degdelta}), one has $\deg \delta(\P(M);t) = \deg h(F(I_M));t) \leq e - 2$.  Thus,  in particular $\reg(F(I_M)) \leq e - 1$.  Furthermore, Theorem \ref{wehope} guarantees that, unless $F(I_M)$ is the polynomial ring, one has $\deg h(F(I_M),t) \geq 2$ and $\reg(I_M) \geq 3$.  As a result,

\begin{Theorem}
\label{regularity}
Let $G$ be a finite $2$-connected simple graph on $[n]$ with $e$ edges.  Let $M$ be the cycle matroid of $G$.  Unless $F(I_M)$ is the polynomial ring, one has
\[
3 \leq \reg(F(I_M)) \leq e - 1.
\]
\end{Theorem}

It is natural to ask, given integers $3 \leq r < e$, if there exists a $2$-connected simple graph $G$ with $e$ edges for which $\reg(F(I_M)) = r$, where $M$ is the cycle matroid of $G$ and $F(I_M)$ is the base ring of $M$.

\begin{Corollary}
\label{regularityCOR}
Let $G$ be a finite disconnected simple graph on $[n]$ with $e$ edges and $G_1, \ldots, G_s$ its connected components.  Let $c_i$ denote the number of cut vertices of $G_i$ for $1 \leq i \leq s$.  Let $M$ be the cycle matroid of $G$.  Unless $F(I_M)$ is the polynomial ring, one has
\[
3 \leq \reg(F(I_M)) \leq e - c - s.
\]
\end{Corollary}

Finally, since the toric ideal of the base ring of a cycle matroid of a finite simple graph is generated by quadratic binomials \cite{Bla}, it follows that no base ring of a cycle matroid has a linear resolution.

\medskip
\noindent
{\bf Acknowledgement.}  This paper is supported by the National Natural Science Foundation of
China (11271275) and by the Foundation of the Priority Academic Program Development of Jiangsu Higher Education Institutions.

\end{document}